\documentclass{article}
\usepackage{amsmath}
\usepackage{amssymb}
\usepackage{amsfonts}

\newtheorem{theorem}{Theorem}

\newtheorem{definition}[theorem]{Definition}

\newtheorem{lemma}[theorem]{Lemma}

\newenvironment{proof}[1][Proof]{\noindent\textbf{#1.} }{\ \rule{0.5em}{0.5em}}
\input{tcilatex}
\begin{document}

\author{Miko\l aj Pep\l o\'{n}ski}
\title{On the existence and multiplicity of solutions for discrete higher
order BVP}
\maketitle

\begin{abstract}
We investigate the existence and multiplicity of solutions for higher order
discrete boundary value problems via critical point theory.
\end{abstract}

\section{Introduction}

\qquad In this paper we consider the discrete boundary value problem of
order $2n$ with Dirichlet type boundary conditions, namely we investigate

\begin{center}
\begin{equation}
\begin{array}{c}
{\small \Delta }^{n}{\small (p(k-n)\Delta }^{n}{\small x(k-n))+(-1)}^{n+1}%
{\small f(k,x(k))=0}\text{ for\ }{\small k\in Z[1,N]} \\ 
\\ 
{\small x(i)=0}\text{ for }{\small i\in Z[1-n,0]\cup Z[N+1,N+n]}%
\end{array}
\label{rownanie wyjsciowe}
\end{equation}
\end{center}

\bigskip 

where $N\geq 2$; $f\in C(Z[1,N]\times 
\mathbb{R}
,%
\mathbb{R}
)$; $p(i)\in 
\mathbb{R}
$ for $i=1-n,...,N$; $\Delta $ denotes the forward difference operator $%
\Delta x(i)=x(i+1)-x(i)$; for any $a,b\in 
\mathbb{Z}
$ we put $\ Z[a,b]=\{a,a+1,...b-1,b\}$. Problem (\ref{rownanie wyjsciowe})
as given above is in a variational form so that we may apply the critical
point theory in order to reach the existence of solutions for which we use
the direct variational method. In order to get the multiplicity of solutions
we combine the mountain pass methodology with the direct variational
approach.

Discrete boundary value problems have attracted a lot of attention recently.
The boundary value problems connected with discrete equations can be tackled
with almost similar methods as their continuous counterparts. The
variational techniques applied for discrete problems include, among others,
the mountain pass methodology, the linking theorem, the Morse theory, the
three critical point, compare with \cite{agrawal}, \cite{caiYu}, \cite%
{TianZeng}, \cite{zhangcheng}, \cite{nonzero}. Moreover, the fixed point
approach is in fact much more prolific in the case of discrete problem, see
for example \cite{FPAgrawal}, \cite{FPYangPing}. However the results
concerning the higher order problems are rather scarce, see \cite{Sol2n}.
Mostly in the literature the second and fourth order problems are considered.

\section{Variational framework}

Solutions are obtained in the space 
\begin{equation*}
E=\{x:Z[1-n,N+n]\rightarrow 
\mathbb{R}
|(\forall i\in Z[1-n,0]\cup Z[N+1,N+n])(x(i)=0)\}
\end{equation*}

\bigskip considered with a norm%
\begin{equation*}
||x||=\left( \dsum\limits_{k=1}^{N}x(k)^{2}\right) ^{\frac{1}{2}}.
\end{equation*}

All functions form $E$ are defined on a finite set, and therefore these are
continuous. The space $E$ can be also considered with the following norm%
\begin{equation}
||x||_{q}=\left( \dsum\limits_{k=1}^{N}|x(k)|^{q}\right) ^{\frac{1}{q}}
\label{norma_r}
\end{equation}%
where $q\geq 1$. Since $E$ has finite dimension any norms $||.||_{1},$ $%
||.||_{2}$ are equivalent, i.e. there exist constants $c,C$ such that
inequality 
\begin{equation*}
c||x||_{1}\leq ||x||_{2}\leq C||x||_{1}
\end{equation*}%
holds for all $x\in E$.

The action functional which we use $J:E\rightarrow 
\mathbb{R}
$ reads 
\begin{equation*}
J(x)=\dsum\limits_{k=1-n}^{N}\left[ \frac{1}{2}p(k)(\Delta
^{n}x(k))^{2}-F(k,x(k))\right]
\end{equation*}%
where $F(k,s)=\dint\limits_{0}^{s}f(k,t)dt$. Critical points of $J$ are in
fact solution to (\ref{rownanie wyjsciowe}) and in turn solutions to (\ref%
{rownanie wyjsciowe}) are precisely critical points to $J$. The solutions
which we investigate are the strong ones. This is in contrast to the
infinite dimensional case, when the critical point theory allows usually for
obtaining weak solutions.

\begin{lemma}
$J$ is continously differentiable in the sense of Gateux; $x_{0}\in E$ is a
solution to (\ref{rownanie wyjsciowe}) if and only if it is critical point
to $J$.
\end{lemma}

\begin{proof}
Fix $x\in E$ and chose an arbitrary direction $h\in E$. Let us define an
auxiliary function $\varphi _{x}:%
\mathbb{R}
\rightarrow 
\mathbb{R}
$ by 
\begin{equation*}
\varphi _{x}(\varepsilon )=J(x+\varepsilon h)=\dsum\limits_{k=1-n}^{N}\left[ 
\frac{1}{2}p(k)(\Delta ^{n}x(k)+\Delta ^{n}(\varepsilon
h)(k))^{2}-F(k,x(k)+(\varepsilon h)(k))\right] .
\end{equation*}

Its derivative at $0$ is equal to derivative of $J$ at point $x$ and
direction $h$%
\begin{eqnarray*}
\varphi _{x}^{\prime }(0) &=&p(k)\dsum\limits_{k=1-n}^{N}\Delta
^{n}x(k)\Delta ^{n}h(k)-\dsum\limits_{k=1-n}^{N}f(k,x(k))h(k)= \\
&&\dsum\limits_{k=1-n}^{N}\left( p(k)\Delta ^{n}x(k)\Delta
^{n}h(k)-f(k,x(k))h(k)\right) .
\end{eqnarray*}%
Using summation by parts formula, we observe that%
\begin{equation*}
\begin{array}{l}
\dsum\limits_{k=1-n}^{N}p(k)\Delta ^{n}x(k)\Delta ^{n}h(k)= \\ 
\lbrack p(k-1)\Delta ^{n}x(k-1)\Delta
^{n-1}h(k)]_{1-n}^{N+1}-\dsum\limits_{k=1-n}^{N}\Delta \left( p(k-1)\Delta
^{n}x(k-1)\right) \Delta ^{n-1}h(k)= \\ 
-\dsum\limits_{k=1-n}^{N}\Delta \left( p(k-1)\Delta ^{n}x(k-1)\right) \Delta
^{n-1}h(k)= \\ 
-\left( [\Delta \left( p(k-2)\Delta ^{n}x(k-2)\right) \Delta
^{n-2}h(k)]_{1-n}^{N+1}-\dsum\limits_{k=1-n}^{N}\Delta ^{2}\left(
p(k-2)\Delta ^{n}x(k-2)\right) \Delta ^{n-2}h(k)\right) = \\ 
\dsum\limits_{k=1-n}^{N}\Delta ^{2}\left( p(k-2)\Delta ^{n}x(k-2)\right)
\Delta ^{n-2}h(k)=(-1)^{n}\dsum\limits_{k=1-n}^{N}\Delta ^{n}\left(
p(k-n)\Delta ^{n}x(k-n)\right) h(k).%
\end{array}%
\end{equation*}%
Therefore $J$ is of class $C^{1}$. Thus 
\begin{equation*}
\varphi _{x}^{\prime }(0)=\dsum\limits_{k=1-n}^{N}\left( (-1)^{n}\Delta
^{n}\left( p(k-n)\Delta ^{n}x(k-n)\right) h(k)-f(k,x(k))h(k)\right)
=(J^{\prime }(x),h)
\end{equation*}%
and $\varphi _{x}^{\prime }(0)=0$ provides that%
\begin{equation*}
\dsum\limits_{k=1-n}^{N}\left[ \Delta ^{n}\left( p(k-n)\Delta
^{n}x(k-n)\right) h(k)+(-1)^{n+1}f(k,x(k))h(k)\right] =0.
\end{equation*}%
Thus the second assertion follows.
\end{proof}

\section{Auxiliary results}

Now we provide some results which will be used in the sequel. Let us recall
some definitions and lemmas.

\begin{definition}
(\cite{mawh})Let $H$ be a Banach space. We say that functional\ $%
J:H\rightarrow 
\mathbb{R}
$ is coercive on $H$, if $\lim_{||x||\rightarrow \infty }J(x)=+\infty $,
where $||.||$ stands for norm in $H$.
\end{definition}

\begin{definition}
(\cite{mawh})Let $J\in C^{1}(H,%
\mathbb{R}
)$, where $H$ is real Banach space. If for any sequence $\{u_{n}\}\subset H$%
, such that $\{J(u_{n})\}$ is bounded and $J^{\prime }(u_{n})\rightarrow 0$
as $n\rightarrow \infty $, $(u_{n})$ possesses a convergent subsequence,
then we say that $J$ satisfies Palais -Smale condition or (PS) condition for
short.
\end{definition}

In search for critical points we will use the following lemmas:

\begin{lemma}
(\cite{mawh}) If the functional $J:H\rightarrow R$, is continous and
coercive, then there exists $x_{0}\in H$ such that $\inf_{x\in
H}J(x)=J(x_{0})$. If $J$ is Gateaux differentable at $x_{0}$, $J^{\prime
}(x_{0})=\theta $.
\end{lemma}

\begin{lemma}
(\cite{mawh})(Mountain-pass lemma) Let $H$ be a real Banach space, and $J\in
C^{1}(H,%
\mathbb{R}
)$ satisfy the (PS) condition. Assume that $x_{0},x_{1}\in H$ and $%
\Omega
\subset H$ is an open set such that $x_{0}\in \Omega $, but $x_{1}\notin
\Omega $. If $max\{J(x_{0}),J(x_{1})\}<inf_{x\in \partial 
\Omega
}J(x)$ , then $c=\inf_{h\in \Gamma }\max_{t\in \lbrack 0,1]}J(h(t))$ is the
critical value of $J$, where 
\begin{equation*}
\Gamma =\{h|h:[0,1]\rightarrow H,\text{h is continuous}%
,h(0)=x_{0},h(1)=x_{1}\}
\end{equation*}%
It means that there exist $x^{\ast }\in H$ such that $J(x^{\ast })=c$ and $%
J^{\prime }(x^{\ast })=\theta $.
\end{lemma}

\bigskip We may also use mountain pass lemma in following variant

\begin{lemma}
Let $H$ be a real Banach space, and $J\in C^{1}(H,%
\mathbb{R}
)$ satisfy the (PS) condition. Assume that $x_{0},x_{1}\in H$ and $%
\Omega
\subset H$ is an open set such that $x_{0}\in \Omega $, but $x_{1}\notin
\Omega $. If $\min \{J(x_{0}),J(x_{1})\}>\sup_{x\in \partial \Omega }J(x)$ ,
then $c=\sup_{h\in \Gamma }\min_{t\in \lbrack 0,1]}J(h(t))$ is the critical
value of $J$. It means that there exist $x^{\ast }\in H$ such that $%
J(x^{\ast })=c$ and $J^{\prime }(x^{\ast })=\theta $.
\end{lemma}

since functional $J$ satisfies assumptions of mountain pass lemma if and
only if $-J$ satisfies assumptions of its second variant, and since critical
points of $J$ and $-J$ are mutually corresponing.

We will also use the following inequality from \cite{Sol2n}

\begin{equation*}
\lambda ||x||^{2}\leq \dsum\limits_{k=1-n}^{N}(\Delta ^{n}x(k))^{2}\leq
4^{n}||x||^{2}
\end{equation*}

\section{Case of nonnegative $p$}

We know that the direct method of the calculus of variations can be
summarized as follows: given a continously differetable coercive functional
we know that it has at least one critical point. Since we know that our
functional is already of class $C^{1}$, it suffice to provide conditions
which will guarantee the coercivity or anty-coercivity of $J.$ In case of
finite dimensional setting we can employ either coercivity or
anty-coercivity of the action functional. Thus we have the following
sequence of lemmas.

\begin{lemma}
\label{lemma_coercive}We assume either that\newline
\textbf{A1} there exists a constant $\alpha >0$, such that $xf(k,x)\leq 0$
for all $k\in Z[1,N]$, and $\left\vert x\right\vert >\alpha $. \newline
or that\newline
\textbf{A2.1} there exist constants $M\geq 0,\alpha <\frac{1}{2}\lambda \min
p(k)$ such that for all $|u|>M$ it holds $F(k,u)\leq \alpha u^{2}$,\newline
or that\newline
\textbf{A2.2} there exist constants $M\geq 0,\alpha \in 
\mathbb{R}
$, $1\leq $ $q<2$ such that for all $|u|>M$ it holds $F(k,u)\leq \alpha
|u|^{q}$. \newline
Then $J$ is coercive and problem (\ref{rownanie wyjsciowe}) has at least one
solution.

\begin{proof}
By \textbf{A1} we see that $x\rightarrow F(k,x)$ are bounded from the above
for $k=1,...,N$. Indeed, fix $k\in Z[1,N]$. We observe that $f(k,x)\leq 0$
for $x>\alpha $. Thus for $x>\alpha $ function $x\rightarrow F(k,x)$ is
nonincreasing. Hence $\limsup\nolimits_{x\rightarrow +\infty }F(k,x)<+\infty
.$ In order to prove the boundness for $x<-\alpha $ we put $h(k,x)=f(k,-x)$
and repeat our reasoning. Since for $x\in \lbrack -\alpha ,\alpha ]$,
function $x\rightarrow F(k,x)$ being continuous is bounded we see that there
exists a constant $m>0$ such that%
\begin{equation*}
\begin{array}{l}
J(x)=\frac{1}{2}\min\limits_{k\in
Z[1-n,N]}p(k)\dsum\limits_{k=1-n}^{N}(\Delta
^{n}x(k))^{2}-\dsum\limits_{k=1-n}^{N}F(k,x(k))\geq \bigskip  \\ 
\frac{1}{2}\lambda \min\limits_{k\in
Z[1-n,N]}p(k)||x||^{2}-\dsum\limits_{k=1-n}^{N}F(k,x(k))\geq \frac{1}{2}%
\lambda \min\limits_{k\in Z[1-n,N]}p(k)||x||^{2}-m.%
\end{array}%
\end{equation*}%
Thereofore $J$ is coercive and so problem (\ref{rownanie wyjsciowe}) has a
solution.

Assuming either \textbf{A2.1} or \textbf{A2.2} we see that%
\begin{equation*}
\begin{array}{l}
\dsum\limits_{k=1-n}^{N}\left[ \frac{1}{2}p(k)(\Delta ^{n}x(k))^{2}-F(k,x(k))%
\right] \geq \dsum\limits_{k=1-n}^{N}\left[ \frac{1}{2}\min p(k)(\Delta
^{n}x(k))^{2}-F(k,x(k))\right] = \\ 
\frac{1}{2}\min p(k)\dsum\limits_{k=1-n}^{N}(\Delta
^{n}x(k))^{2}-\dsum\limits_{k=1-n}^{N}F(k,x(k))\geq  \\ 
\frac{1}{2}\min p(k)\dsum\limits_{k=1-n}^{N}(\Delta
^{n}x(k))^{2}-\dsum\limits_{k=1-n}^{N}\alpha |x(k)|^{q}=\frac{1}{2}\min
p(k)\lambda ||x||^{2}-\alpha ||x||_{q}^{q}\geq  \\ 
\frac{1}{2}\min p(k)\lambda ||x||^{2}-\alpha C^{q}||x||^{q}%
\end{array}%
\end{equation*}%
Where $||.||_{q}$ is defined by (\ref{norma_r}) and $C$ is constant such
that for every $x\in E$ inequality $||x||_{q}\leq C||x||$ holds. When $1\leq
q<2$ we see that the expression 
\begin{equation*}
\frac{1}{2}\min p(k)\lambda ||x||^{2}-\alpha C^{q}||x||^{q}
\end{equation*}%
approaches $+\infty $ as $||x||\rightarrow \infty $. When $q=2$ and $\alpha <%
\frac{1}{2}\lambda \min p(k)$, we see that 
\begin{equation*}
(\frac{1}{2}\min p(k)\lambda -\alpha )||x||^{2}\rightarrow +\infty 
\end{equation*}%
since $(\frac{1}{2}\min p(k)\lambda -\alpha )>0.$
\end{proof}
\end{lemma}

The existence results we can also get in case when the action functional is
anti-coercive. Indeed, we have the following result similar in spirit to
Lemma \ref{lemma_coercive}.

\begin{lemma}
\label{lemma_anticoercive}We assume either that\newline
\textbf{A3.1} there exist constants $M\geq 0,\alpha >\frac{1}{2}\max
p(k)4^{n}$ such that for all $|u|>M$ it holds $F(k,u)\geq \alpha u^{2}$%
\newline
or that\newline
\textbf{A3.2} there exist constants $M\geq 0,\alpha >0$, $q>2$ such that for
all $|u|>M$ it holds $F(k,u)\geq \alpha |u|^{q}$.\newline
Then $J$ is anti-coercive and problem (\ref{rownanie wyjsciowe}) has at
least one solution.
\end{lemma}

\begin{proof}
Note that%
\begin{equation*}
\begin{array}{l}
J(x)\leq \ \frac{1}{2}\max
p(k)4^{n}||x||^{2}-\dsum\limits_{k=1-n}^{N}F(k,x(k))\leq  \\ 
\frac{1}{2}\max p(k)4^{n}||x||^{2}-\alpha \dsum\limits_{k=1-n}^{N}|x(k)|^{q}=%
\frac{1}{2}\max p(k)4^{n}||x||^{2}-\alpha (||x||_{q})^{q}\leq  \\ 
\\ 
\frac{1}{2}\max p(k)4^{n}||x||^{2}-\alpha c^{q}|||x||^{q}.%
\end{array}%
\end{equation*}

If $q>2$ we see that 
\begin{equation*}
\frac{1}{2}\max p(k)4^{n}||x||^{2}-\alpha c^{q}|||x||^{q}
\end{equation*}%
approaches $-\infty $ as $||x||\rightarrow \infty $ since $\alpha \geq 0$.
When $q=2$ and $\alpha >\frac{1}{2}\max p(k)4^{n}$ we see that

\begin{equation*}
J(x)\leq \ \frac{1}{2}\max p(k)\cdot 4^{n}\cdot
||x||^{2}-\dsum\limits_{k=1-n}^{N}F(k,x(k))\leq (\frac{1}{2}\max p(k)\cdot
4^{n}-\alpha )||x||^{2}\rightarrow -\infty
\end{equation*}
\end{proof}

\section{\protect\bigskip Case of p non-positive}

In this case the existence follows with similar methods as in the previous
case. We provide the results with only sketched proofs.

\begin{lemma}
\label{lemma_anticoercive_nonpos}We assume either that\newline
\textbf{B1} there exists a constant $\alpha >0$, such that $xf(k,x)\geq 0$
for all $k\in Z[1,N]$, and $\left\vert x\right\vert >\alpha $. \newline
or that\newline
\textbf{B2.1} there exist constants $M\geq 0,\alpha >\frac{1}{2}4^{n}\max
p(k)$ such that for all $|u|>M$ it holds $F(k,u)\geq \alpha u^{2}$,\newline
or that\newline
\textbf{B2.2} there exist constants $M\geq 0,\alpha \in 
\mathbb{R}
$, $1\leq $ $q<2$ such that for all $|u|>M$ it holds $F(k,u)\geq \alpha
|u|^{q}$. \newline
Then $J$ is anticoercive and problem (\ref{rownanie wyjsciowe}) has at least
one solution.
\end{lemma}

\begin{proof}
We consider functional

\begin{equation*}
\begin{array}{l}
K(x)=-J(x)=-\dsum\limits_{k=1-n}^{N}\left[ \frac{1}{2}p(k)(\Delta
^{n}x(k))^{2}-F(k,x(k))\right] = \\ 
\dsum\limits_{k=1-n}^{N}\left[ \frac{1}{2}(-p(k))(\Delta
^{n}x(k))^{2}-(-F(k,x(k)))\right] .%
\end{array}%
\end{equation*}%
We see that $K$,\textbf{\ }with $p^{\prime }=-p$ and $G=-F$ satisfiy \textbf{%
A1}. Since in finite dimensional space every $x$ is a critical point of $J$
if and only if it is a critical point of $K$ we get assertion.

Assuming either \textbf{B2.1 }or\textbf{\ B2.2} we see that%
\begin{equation*}
\begin{array}{l}
\dsum\limits_{k=1-n}^{N}\left[ \frac{1}{2}p(k)(\Delta ^{n}x(k))^{2}-F(k,x(k))%
\right] \leq \frac{1}{2}\max p(k)\dsum\limits_{k=1-n}^{N}(\Delta
^{n}x(k))^{2}-\dsum\limits_{k=1-n}^{N}F(k,x(k))\bigskip \leq  \\ 
\frac{1}{2}\max p(k)4^{n}||x||^{2}-\dsum\limits_{k=1-n}^{N}F(k,x(k))\leq 
\frac{1}{2}\max p(k)4^{n}||x||^{2}-\alpha c^{q}||x||^{q}.%
\end{array}%
\end{equation*}%
When $1\leq q<2$ we see, that right side of inequality approaches $-\infty $%
. If $q=2$ and $\alpha >\frac{1}{2}\max p(k)4^{n}$, then%
\begin{equation*}
\frac{1}{2}\max p(k)4^{n}||x||^{2}-\alpha ||x||^{2}
\end{equation*}%
also approaches $-\infty $, since $\frac{1}{2}\max p(k)4^{n}-\alpha <0$.
\end{proof}

\begin{lemma}
\label{lemma_coercive_nonpos}We assume either that\newline
\textbf{B3.1} there exist constants $M\geq 0,\alpha <\frac{1}{2}\min
p(k)\lambda $ such that for all $|u|>M$ it holds $F(k,u)\leq \alpha u^{2}$%
\newline
or that\newline
\textbf{B3.2} there exist constants $M\geq 0,\alpha <0$, $q>2$ such that for
all $|u|>M$ it holds $F(k,u)\leq \alpha |u|^{q}$,\newline
Then $J$ is coercive and problem (\ref{rownanie wyjsciowe}) has at least one
solution.
\end{lemma}

\begin{proof}
\textbf{\ }Note that

\begin{equation*}
\begin{array}{l}
J(x)=\dsum\limits_{k=1-n}^{N}\left[ \frac{1}{2}p(k)(\Delta
^{n}x(k))^{2}-F(k,x(k))\right] \geq \dsum\limits_{k=1-n}^{N}\left[ \frac{1}{2%
}\min p(k)(\Delta ^{n}x(k))^{2}-F(k,x(k))\right]  \\ 
=\frac{1}{2}\min p(k)\dsum\limits_{k=1-n}^{N}(\Delta
^{n}x(k))^{2}-\dsum\limits_{k=1-n}^{N}F(k,x(k))\geq \frac{1}{2}\min
p(k)\lambda ||x||^{2}-\alpha C^{q}||x||^{q}%
\end{array}%
\end{equation*}

When $q>2$ then since $\alpha <0$, left side of inequality approaces $%
+\infty $. When $q=2$, and $\alpha <\frac{1}{2}\min p(k)\lambda ,$ it also
approaces $+\infty $.
\end{proof}

\section{Case of arbitrary $p$}

When function $p$ has arbitrary sign we may also use the arguments applied
before. For example by inequality

\begin{equation*}
\dsum\limits_{k=1-n}^{N}\left[ \frac{1}{2}p(k)(\Delta ^{n}x(k))^{2}-F(k,x(k))%
\right] \leq \dsum\limits_{k=1-n}^{N}\left[ \frac{1}{2}\max p(k)(\Delta
^{n}x(k))^{2}-F(k,x(k))\right]
\end{equation*}%
it follows with each of the assumptions \textbf{A3.1, A3.2, B2.1, B2.2 }that
functional\textbf{\ }$J$ is anti-coercive, and therefore (\ref{rownanie
wyjsciowe}) has a solution.

Similarly by inequality 
\begin{equation*}
\bigskip \dsum\limits_{k=1-n}^{N}\left[ \frac{1}{2}\min p(k)(\Delta
^{n}x(k))^{2}-F(k,x(k))\right] \leq \dsum\limits_{k=1-n}^{N}\left[ \frac{1}{2%
}p(k)(\Delta ^{n}x(k))^{2}-F(k,x(k))\right]
\end{equation*}%
it follows with each of the assumptions \textbf{A2.1, A2.2, B3.1, B3.2 }that%
\textbf{\ }functional\textbf{\ }$J$ is coercive, and therefore (\ref%
{rownanie wyjsciowe}) has a solution.

\bigskip

Moreover, for the purpose of the existence of at least one solution we may
use one of the following conidtions:

\textbf{C1} there exists $\alpha >\frac{1}{2}\max p(k)4^{n}$\ such that $%
\lim \inf_{|u|\rightarrow \infty }\frac{F(k,u)}{u^{2}}>\alpha $ uniformely
in $k\in Z[1,N]$,

\textbf{C2} there exists $\alpha >\frac{1}{2}\max p(k)4^{n}$\ such that $%
\lim \inf_{|u|\rightarrow \infty }\frac{F(k,u)}{u^{2}}=\alpha $, uniformely
in $k\in Z[1,N]$,

\textbf{C3} there exists $\alpha >\frac{1}{2}\max p(k)4^{n}$\ such that $%
\lim \inf_{|u|\rightarrow \infty }\frac{F(k,u)}{u^{2}}\geq \alpha $,
uniformely in $k\in Z[1,N]$,

\bigskip \textbf{C4} there exists $\alpha >\frac{1}{2}\max p(k)4^{n}$ such
that $\lim_{|u|\rightarrow \infty }(F(k,u)-\alpha u^{2})=\infty $,
uniformely in $k\in Z[1,N]$,

\textbf{D1} there exists $\alpha >0$, $q>2$\ such that $\lim
\inf_{|u|\rightarrow \infty }\frac{F(k,u)}{|u|^{q}}>\alpha $, uniformely in $%
k\in Z[1,N]$,

\textbf{D2} there exists $\alpha >0$, $q>2$\ such that $\lim
\inf_{|u|\rightarrow \infty }\frac{F(k,u)}{|u^{|q}}=\alpha $, uniformely in $%
k\in Z[1,N]$,

\textbf{D3} there exists $\alpha >0$, $q>2$\ such that $\lim
\inf_{|u|\rightarrow \infty }\frac{F(k,u)}{|u|^{q}}\geq \alpha $, uniformely
in $k\in Z[1,N]$,

\bigskip \textbf{D4} there exists $\alpha >0$, $q>2$ such that $%
\lim_{|u|\rightarrow \infty }(F(k,u)-\alpha |u|^{q})=\infty $, uniformely in 
$k\in Z[1,N]$,

\textbf{E1} \bigskip there exists $\alpha <\frac{1}{2}\min p(k)\lambda $\
such that $\limsup_{||u||\rightarrow \infty }\frac{F(k,u)}{u^{2}}<\alpha $,
uniformely in $k\in Z[1,N]$,

\textbf{E2} \bigskip there exists $\alpha <\frac{1}{2}\min p(k)\lambda $\
such that $\limsup_{||u||\rightarrow \infty }\frac{F(k,u)}{u^{2}}=\alpha $,
uniformely in $k\in Z[1,N]$,

\textbf{E3} \bigskip there exists $\alpha <\frac{1}{2}\min p(k)\lambda $\
such that $\limsup_{||u||\rightarrow \infty }\frac{F(k,u)}{u^{2}}\leq \alpha 
$, uniformely in $k\in Z[1,N]$,

\textbf{E4} there exists $\alpha <\frac{1}{2}\min p(k)\lambda $ such that $%
\lim_{|u|\rightarrow \infty }(F(k,u)-\alpha u^{2})=-\infty $, uniformely in $%
k\in Z[1,N]$,

\textbf{F1} \bigskip there exists $\alpha <0$, $q>2$ such that $%
\limsup_{||u||\rightarrow \infty }\frac{F(k,u)}{|u|^{q}}<\alpha $,
uniformely in $k\in Z[1,N]$,

\textbf{F2} there exists $\alpha <0$, $q>2$\ such that $\limsup_{||u||%
\rightarrow \infty }\frac{F(k,u)}{|u|^{q}}=\alpha $, uniformely in $k\in
Z[1,N]$,

\textbf{F3} \bigskip there exists $\alpha <0$, $q>2$ \ such that $%
\limsup_{||u||\rightarrow \infty }\frac{F(k,u)}{|u|^{q}}\leq \alpha $,
uniformely in $k\in Z[1,N]$,

\textbf{F4} there exists $\alpha <0$, $q>2$ such that $\lim_{|u|\rightarrow
\infty }$ \ $(F(k,u)-\alpha |u|^{q})=-\infty $, uniformely in $k\in Z[1,N]$.

We have the following

\begin{lemma}
Assume any of the above conditions. Then problem (\ref{rownanie wyjsciowe})
has at least one solution.
\end{lemma}

\begin{proof}
When \bigskip \textbf{C1} is satisfied we use the definition of the lower
limit in order to get%
\begin{equation*}
(\exists \Delta >0)(\forall u)(|u|>\Delta \Rightarrow \frac{F(k,u)}{u^{2}}%
>\alpha )
\end{equation*}%
and further%
\begin{equation*}
(\forall u)(|u|>\Delta \Rightarrow F(k,u)>\alpha u^{2})
\end{equation*}%
Thus there exist $\alpha >\frac{1}{2}\max p(k)4^{n}$, $\Delta >0$ such that
for $|u|>\Delta $, $F(k,u)>\alpha u^{2}$. So \textbf{A3.1 }holds and the
assertion follows by Lemma (\ref{lemma_anticoercive}).When \textbf{C2 }holds
it suffice to apply \textbf{C1} with $\alpha ^{\prime }=\frac{\alpha }{2}$. 
\textbf{C3 }follows by any of \textbf{C1} or \textbf{C2}. Assuming \bigskip 
\textbf{C4} we fix $M>0$ and find $\Delta >0$ such that for all $|u|>\Delta $
it holds $F(k,u)-\alpha u^{2}>M$. Thus there exist $\alpha >\frac{1}{2}\max
p(k)4^{n},\Delta $ such that $|u|>\Delta $, $F(k,u)>\alpha u^{2}+M$. So
again \textbf{A3.1 }holds and the assertion follows by Lemma (\ref%
{lemma_anticoercive}).

\bigskip

When \ \textbf{D1 }is satisfied we use the definition of lower limit in
order to get 
\begin{equation*}
(\exists \Delta >0)(\forall u)(|u|>\Delta \Rightarrow \frac{F(k,u)}{|u|^{q}}%
>\alpha )
\end{equation*}%
and then%
\begin{equation*}
(\forall u)(|u|>\Delta \Rightarrow F(k,u)>\alpha |u|^{q})
\end{equation*}%
Thus there exist $\alpha >0$, $\Delta >0$ such that for $|u|>\Delta $, $%
F(k,u)>\alpha |u|^{q}$ where $q>2$. Hence \textbf{A3.2 }holds and the
assertion follows by Lemma (\ref{lemma_anticoercive}) In similar way as in
previous case we obtain, that \textbf{D2, D3} provide coercitivity of $J.$%
Assuming \textbf{D4} we fix $M>0$ and find $\Delta >0$ such that for all $%
|u|>\Delta $ it holds $F(k,u)-\alpha u^{2}>M$. Thus there exist $\alpha >%
\frac{1}{2}\max p(k)4^{n},\Delta $ such that $|u|>\Delta $, $F(k,u)>\alpha
u^{2}+M$. So again \textbf{A3.2 }holds and the assertion follows by Lemma (%
\ref{lemma_anticoercive})

Similar as in previous cases we obtain that \textbf{\ }there exist $\alpha <%
\frac{1}{2}\lambda \min p(k)$, $\Delta >0$ such that for $|u|>\Delta $, $%
F(k,u)<\alpha u^{2}$. So \textbf{B3.1} holds and the assertion follows by
Lemma (\ref{lemma_coercive_nonpos})

\textbf{E4. }Fix $M<0$. Using definition of upper limit, there exists $%
\Delta >0$ such that for each $u$, if $|u|>\Delta $, then $F(k,u)-\alpha
u^{2}<M$. Therefore assumption \textbf{B3.1} is satisfied and functional\ J
is coercive.

When \textbf{F1} is satisfied we use definition of upper limit in order to
get%
\begin{equation*}
(\exists \Delta >0)(\forall u)(|u|>\Delta \Rightarrow \frac{F(k,u)}{|u|^{q}}%
<\alpha )
\end{equation*}

and then%
\begin{equation*}
(\forall u)(|u|>\Delta \Rightarrow F(k,u)<\alpha |u|^{q}
\end{equation*}%
So \textbf{B3.2 }is satisfied. To prove \textbf{F4} fix $M<0.$ Using the
definition of upper limit we obtain, that there exists $\Delta >0$ such that
for each $u$, if $|u|>\Delta $, then $F(k,u)<M+\alpha |u|^{q}$. Therefore
assumption \textbf{B3.2} is satisfied and functional $J$ is coercive.
\end{proof}

\section{Existence of a second solution}

In this section we follow the reasoning applied in \cite{Sol2n}. We have the
following theorem.

\begin{theorem}
\label{tw_2_rozw}Assume that any of the conditions \textbf{A3.1, A3.2, B2.1,
B2.2} holds\textbf{,} and that\textbf{\ }$\max_{k\in
Z[1,N]}\lim_{u\rightarrow 0}\frac{f(k,u)}{u}\leq c<\min p(k)\lambda $\textbf{%
. }Then(\ref{rownanie wyjsciowe}) has at least two solutions.
\end{theorem}

\begin{proof}
Since $J$ is anti-coercive and since $E$ is finite dimensional, it follows
that $J$ satisfies the (PS) condition. Let $\varepsilon =\frac{\min
p(k)\lambda -c}{2}$. There exists $\delta >0$ such that $\ $for $|u|<\delta $
we have $\frac{f(k,u)}{u}\leq c+\varepsilon =\frac{\min p(k)\lambda +c}{2}.$
Hence for $|u|<\delta $, $\int_{0}^{u}f(k,s)ds\leq \frac{\min p(k)\lambda +c%
}{2}\int_{0}^{u}sds=\frac{\min p(k)\lambda +c}{4}u^{2}$ and

\begin{equation*}
\begin{array}{l}
J(x)\geq \frac{1}{2}\lambda ||x||^{2}\min\limits_{k\in
Z[1-n,N]}p(k)-\dsum\limits_{k=1-n}^{N}\frac{\min p(k)\lambda +c}{4}x(k)^{2}=
\\ 
\frac{1}{2}\lambda ||x||^{2}\min\limits_{k\in Z[1-n,N]}p(k)-\frac{\min
p(k)\lambda +c}{4}||x||^{2}= \\ 
(\frac{1}{2}\lambda \min\limits_{k\in Z[1-n,N]}p(k)-\frac{\min p(k)\lambda +c%
}{4})||x||^{2}=\frac{\min p(k)\lambda -c}{4}||x||^{2}>0%
\end{array}%
\end{equation*}

Therefore $J(x)\geq \frac{\min p(k)\lambda -c}{4}\delta ^{2}>0=J(\theta )$
for $x\in \partial \{x\in E:||x||<\delta \}$. Since $\lim_{||x||\rightarrow
\infty }J(x)\rightarrow -\infty $ we easily find $x_{0}\in \Omega =$ $%
E\backslash \{x\in E:||x||\leq \delta \}$ such that $J(x_{0})<0$. In a
consequence $\Omega $, $\theta $, $x$, $x_{0}$ satisfy the assumptions of
the mountain pass lemma. Thus there exists a critical point $\overline{x}$
such that $J(\overline{x})=\inf_{h\in \Gamma }\max_{t\in \lbrack
0,1]}J(h(t)) $. We know by anti-coercivity that there exists $x^{\ast }\in E$
such that $J(x^{\ast })=\max_{x\in E}J(x)$. When $x^{\ast }\neq \overline{x}$
we reach the assertion of the theorem.

Suppose that $x^{\ast }=\overline{x}$. It means that $J(x^{\ast
})=\inf_{h\in \Gamma }\max_{t\in \lbrack 0,1]}J(h(t)).$Hence for any
function $h\in \Gamma $, $\max_{t\in \lbrack 0,1]}J(h(t))=J(x^{\ast })$.
Indeed, for any $h\in \Gamma $ we have$\ J(x^{\ast })\geq $ $\max_{t\in
\lbrack 0,1]}J(h(t))$ since $J(x^{\ast })=\max_{x\in E}J(x)$ and $J(x^{\ast
})\leq $ $\max_{t\in \lbrack 0,1]}J(h(t)$ by definition of the infimum.
Since $N>1$, the space $E\backslash \{x^{\ast }\}$ is path-connected being
homeomorphic with $%
\mathbb{R}
^{N}\backslash \{c\}$, $c\in 
\mathbb{R}
^{N}$. Hence, there exists a function $h_{0}\in \Gamma $ such that $%
h_{0}(t)\neq x^{\ast }$ for $t\in \lbrack 0,1]$. Since $\max_{t\in \lbrack
0,1]}J(h_{0}(t))=J(x^{\ast })$\ it follows that there exists $t_{0}\in (0,1)$%
\ such that $J(h_{0}(t_{0}))=\max_{x\in E}J(x)$\ and $h_{0}(t_{0})\neq
x^{\ast }$\ by the definition of $h_{0}$. Thus $h_{0}(t_{0})$\ is a critical
point different than $x^{\ast }$.
\end{proof}

\bigskip

Using the second variant of the mountain pass theorem and using the
methodology employed in the proof of Theorem (\ref{tw_2_rozw}) we reach the
following result.

\begin{theorem}
\bigskip Assume\ that any of the conditions \textbf{B3.1, B3.2, A2.1, A2.2}
holds\textbf{,} and that\textbf{\ }$\min_{k\in Z[1,...N]}\lim_{u\rightarrow
0}\frac{f(k,u)}{u}\geq c>\max p(k)4^{n}$\textbf{\ }Then(\ref{rownanie
wyjsciowe}) has at least two solutions.

\begin{proof}
Taking $\varepsilon =\frac{c-\max p(k)}{2}$ we obtain, that there exists $%
\delta >0$ such that $\ $for $|u|<\delta $ we have $\frac{f(k,u)}{u}\geq
c+\varepsilon =\frac{\max p(k)4^{n}+c}{2}.$ Hence for $|u|<\delta $, $%
\int_{0}^{u}f(k,s)ds\leq \frac{\max p(k)4^{n}+c}{2}\int_{0}^{u}sds=\frac{%
\max p(k)4^{n}+c}{4}u^{2}$ and%
\begin{equation*}
J(x)\leq ||x||^{2}\cdot \frac{\max p(k)4^{n}-c}{4}<0
\end{equation*}%
For $x\in \partial \{x\in E:||x||<\delta \}$. Since $\lim_{||x||\rightarrow
\infty }J(x)\rightarrow +\infty $ we easily find $x_{0}\in \Omega =$ $%
E\backslash \{x\in E:||x||\leq \delta \}$ such that $J(x_{0})>0$. In a
consequence $\Omega $, $\theta $, $x$, $x_{0}$ satisfy the assumptions of
the second variant of mountain pass lemma. Thus there exists a critical
point $\overline{x}$ such that $J(\overline{x})=\sup_{h\in \Gamma
}\min_{t\in \lbrack 0,1]}J(h(t))$. We know by coercivity that there exists $%
x^{\ast }\in E$ such that $J(x^{\ast })=\min_{x\in E}J(x)$. Using the same
reasoning as in previous case, we obtain existence of second solution.
\end{proof}
\end{theorem}

\end{document}